\documentclass[12pt]{article}
\usepackage{amsmath,amssymb,amsfonts,amsthm}
\usepackage[margin=1in]{geometry}
\usepackage{verbatim}
  
\usepackage{graphicx}
\usepackage{subcaption}

\theoremstyle{plain}

\newtheorem{theorem}{Theorem}[section]
\newtheorem{proposition}[theorem]{Proposition}
\newtheorem{corollary}[theorem]{Corollary}

\usepackage{chngcntr}
\counterwithin{table}{section}

\theoremstyle{remark}

\newtheorem*{remark}{Remark}

\title{Irregular triangulations of complete graphs on $12s$ vertices in orientable surfaces}
\author{Timothy Sun\\Columbia University}
\date{}

\newcommand{\Z}{\mathbb{Z}}

\begin{document}

\maketitle

\begin{abstract}
We present a family of index 1 abelian current graphs whose derived embeddings can be modified into triangulations of $K_{12s}$ for $s \geq 4$. Our construction is significantly simpler than previous methods for finding genus embeddings of $K_{12s}$, which utilized either large index or nonabelian groups.
\end{abstract}

\section{Introduction}

Showing that the complete graph $K_n$ can be embedded in the surface of genus $$\left\lceil \frac{(n-3)(n-4)}{12}\right\rceil$$
proves the Map Color Theorem~\cite{Ringel-MapColor}, which characterizes the chromatic number of all surfaces of positive genus. ``Case 0,'' which consists of all graphs $K_n$ where $n \equiv 0 \pmod{12}$, is traditionally treated separately from the remaining residues due to its resolution using \emph{nonabelian} current graphs. Namely, Terry~\emph{et al.}~\cite{Terry-Case0} employed the representation theory of finite fields to label a current graph with a nonabelian group of order $n = 12s$. 

Since irreducible polynomials over finite fields are not explicitly known in general, their approach may be considered nonconstructive. To remedy this, Pengelley and Jungerman~\cite{Pengelley-Index4} announced a solution using index 4 current graphs with the cyclic group $\Z_{12s}$, but gave details only for $s=1$ and $s \equiv 0 \pmod{8}$. Korzhik~\cite{Korzhik-Index4} improved on their method, giving four families of current graphs, one for each residue $s \bmod{4}$, for $s=4$ and $s \geq 6$. Unfortunately, the aformentioned current graphs in this high index regime are even more complicated than those of Terry \emph{et al.}~\cite{Terry-Case0}, despite the much simpler group involved. 

Mahnke~\cite{Mahnke-Case0} proved that ``there really is no index 1 solution for all $n \equiv 0 \pmod{12}$ with an Abelian group.'' The purpose of this paper is to give such a solution.

\section{Graph embeddings with current graphs}

For more information on topological graph theory, especially a formal treatment of current graphs, see Gross and Tucker~\cite{GrossTucker}. Proofs of correctness for the combinatorial techniques described here can be found in Ringel~\cite{Ringel-MapColor}. We consider only embeddings in orientable surfaces---the analogous statement for the nonorientable genus of $K_{12s}$ is swiftly handled in \S8.2 of Ringel~\cite{Ringel-MapColor}. 

A consequence of Euler's formula is that the genus of a simple graph $G$ is at least
$$\frac{|E(G)|-3|V(G)|+6}{6},$$
and this value is attained by a triangulation of $G$, i.e., an embedding where every face is triangular. Our main result is thus a new construction for triangulations of $K_{12s}$, which are embeddings in the surface of genus $(4s-1)(3s-1)$. The above expression also implies, roughly speaking, that a handle corresponds to six edges in a triangulation. This relationship is crucial for so-called ``additional adjacency'' solutions in the proof of the Map Color Theorem, where a triangulation of a nearly complete graph is augmented using handles to obtain a minimum genus embedding of a complete graph. Ringel \emph{et al.} were successful in finding non-triangular embeddings in this manner, but in the present paper we apply an additional adjacency step that results in a triangulation of a complete graph. 

A cellular embedding of a graph on a surface can be described combinatorially by a \emph{rotation system}, where each vertex $v$ is assigned a \emph{rotation}, a cyclic permutation of the edge ends incident with $v$. In the case of simple graphs, we only need to list the neighbors of $v$. The faces of the embedding can be traced out from the rotation system, but here, only the following specialization for triangulations is necessary:

\begin{proposition}
An embedding is a triangulation if and only if the corresponding rotation system satisfies \emph{Rule R*}, that is, for all edges $(i,k)$, if the rotation at $i$ is of the form
$$i. \,\, \dotsc j \, k \, l \dotsc,$$
then the rotation at $k$ is of the form
$$k. \,\, \dotsc l \, i \, j \dotsc.$$
\label{prop-ruler}
\end{proposition}

A \emph{current graph} is a directed graph embedded in a surface, whose arcs are labeled using \emph{currents}, elements from a group $\Gamma$. The \emph{index} of the current graph is the number of faces in its embedding. For the remainder of the paper, except in Appendix~\ref{sec-app}, we only consider index 1 current graphs and take $\Gamma$ to be cyclic groups of the form $\Z_{12s-4}$. Let the \emph{excess} of a vertex be the sum of its incoming currents minus the sum of the outgoing currents. The name of our combinatorial tool comes from the fact that a vertex typically satisfies \emph{Kirchhoff's current law} (KCL), i.e. its excess is zero.\footnote{If $\Gamma$ were nonabelian, as in the case of Terry \emph{et al.}~\cite{Terry-Case0}, there is the additional complication where the order of summation for KCL is determined the rotation at that vertex.}

There is no actual contradiction with the result of Mahnke~\cite{Mahnke-Case0}. Her proof, and all aforementioned constructions, deal with embeddings that are regular coverings of current graphs, i.e. ones where the derived graph is exactly $K_{12s}$. Our departure from previous approaches is that we start with an embedding of $K_{12s}-K_4$ and add the six missing edges using one handle, breaking some of the symmetry along the way. Triangulations of $K_{12s}-K_4$ were known to Jungerman and Ringel~\cite{JungermanRingel-Minimal}, but for our purposes, we need the associated current graphs to have a specific structure. 

These current graphs, unlike those found in previous work on Case 0~\cite{Terry-Case0, Pengelley-Index4, Korzhik-Index4,Mahnke-Case0}, have \emph{vortices}, vertices where KCL is not satisfied. For each vortex, we label each incident face corner with a letter. We only examine current graphs that satisfy the following ``construction principles," which we paraphrase from Jungerman and Ringel~\cite{JungermanRingel-Minimal}:

\begin{enumerate}
\item[(C1)] Every vertex has degree 3, except vortices, which have degree 1 or 2.
\item[(C2)] The embedding of $D$ has one face, i.e. the current graph is of index 1.
\item[(C3)] For every element $\gamma \in \Gamma\setminus\{0\}$ not of order 2, exactly one of $\gamma$ or $-\gamma$ appears as a current. 
\item[(C4)] KCL holds at every vertex of degree 3. 
\item[(C5)] The order 2 element of $\Gamma$ appears on an arc incident with a vertex of degree 1. 
\item[(C6)] The excess of a vortex of degree $1$ generates $\Gamma$. 
\item[(C7)] The currents incident with a vortex of degree $2$ are both odd, and the excess of that vortex generates the index 2 subgroup of $\Gamma$, i.e., the subgroup of even numbers. 
\end{enumerate}

The one face boundary can be expressed as a cyclic sequence that alternates between arcs and face corners. Its \emph{log} is obtained by replacing arcs with their currents, and corners with their vortex letters, if any. If an arc labeled $\gamma$ is traversed in the the same direction as its orientation, the arc is replaced by $\gamma$; if it is traversed in the reverse direction, the arc is replaced by $-\gamma$. By (C5), the order 2 element appears twice consecutively in the log, which by convention, we condense into one instance. The standard interpretation is that the derived graph actually has a ``doubled 1-factor'' which is then supressed.

The log of a current graph describes a symmetric embedding of a certain derived graph, whose vertices are the elements of $\Gamma$ and the vortex letters. For example, the log of the current graph in Figure~\ref{fig-case3} is
$$\arraycolsep=4.5pt
\begin{array}{rrrrrrrrrrrrrrrrrrrrrrrrrrrrrrrrrrrrrrrr}
3 & x & 1 & y & 31 & z & 29 & 24 & 20 & 2 & 21 & 25 & 15 & 6 & 16 & 22 & 7 & 28 & \dots \\ 8 & 5 & 23 & 17 & 10 & 26 & 9 & 14 & 12 & 4 & 11 & 13 & w & 19 & 30 & 18 & 27
\end{array}$$

\begin{figure}[!ht]
    \centering
    \includegraphics[scale=0.9]{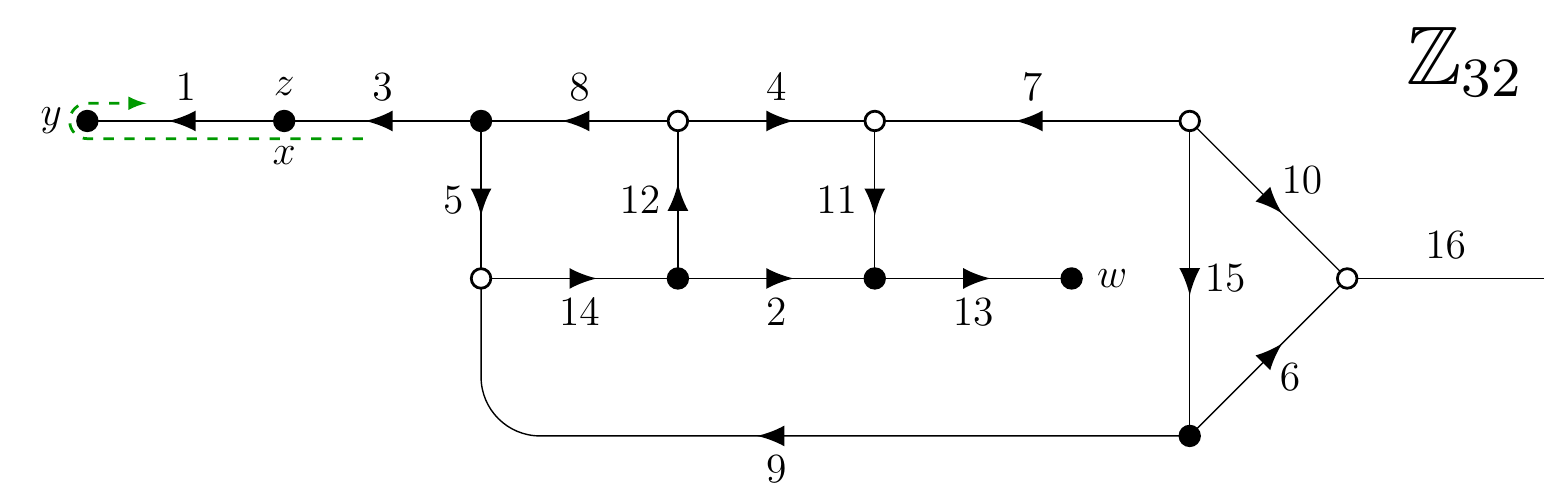}
    \caption{A current graph whose derived graph is $K_{36}-K_4$. Solid and hollow vertices represent clockwise and counterclockwise rotations, respectively.}
    \label{fig-case3}
\end{figure}

The log describes the rotation at vertex 0 (i.e. the identity element of $\Gamma$) and furthermore  determines the rest of the rotation system. We obtain the rotation at vertex $k \in \Gamma$ by adding $k$ to each of the elements of $\Gamma$ in the log. For the vortices of degree 1, we leave their letters unchanged, but for vortices of degree 2, that is, the one corresponding to the letters $x$ and $z$ in this example, we switch the letters' positions if $k$ is odd. A partial picture of the rotation system would therefore look like
$$\begin{array}{crrrrrrrrrrrrrrrrrrrrrrrrrrrrrrr}
0. & 3 & x & 1 & y & 31 & z & 29 & 24 & 20 & \cdots \\
1. & 4 & z & 2 & y & 0 & x & 30 & 25 & 21 \\
2. & 5 & x & 3 & y & 1 & z & 31 & 26 & 22 \\
3. & 6 & z & 4 & y & 2 & x & 0 & 27 & 23\\
\vdots & & & & & & & & & & \ddots
\end{array}$$
The rotations at the lettered vertices are ``manufactured'' such that the embedding near these vertices is triangular, which we do with the help of Proposition~\ref{prop-ruler}. For example, the rotation at vertex $x$ is of the following form:
$$\begin{array}{crrrrrrrrrrrrrrrrrrrrrrrrrrrrrrr}
x. & \dots & 31 & 30 & 1 & 0 & 3 & 2 & 5 & 4 & \dots
\end{array}$$
Because the current graph satisfies the construction principles, the entire embedding is triangular, and the rotations at the lettered vertices form proper cyclic permutations. This is a triangulation of $K_{36}-K_{4}$, as the only missing adjacencies are between lettered vertices. We will show in the next section how we can add the remaining edges using one handle to get a triangulation of $K_{36}$ as desired. 

\section{The genus of $K_{12s}$}

We now present our family of index 1 abelian current graphs for $K_{12s}-K_4$, $s \geq 4$. These yield, with a few small special cases, the simplest known proof of Case 0 of the Map Color Theorem.

\begin{theorem}
There exists an orientable triangulation of $K_{12s}$ for all $s \geq 1$.  
\end{theorem}

\begin{proof}

\begin{figure}[!ht]
    \begin{subfigure}[b]{\textwidth}
        \centering
        \includegraphics[width=\textwidth]{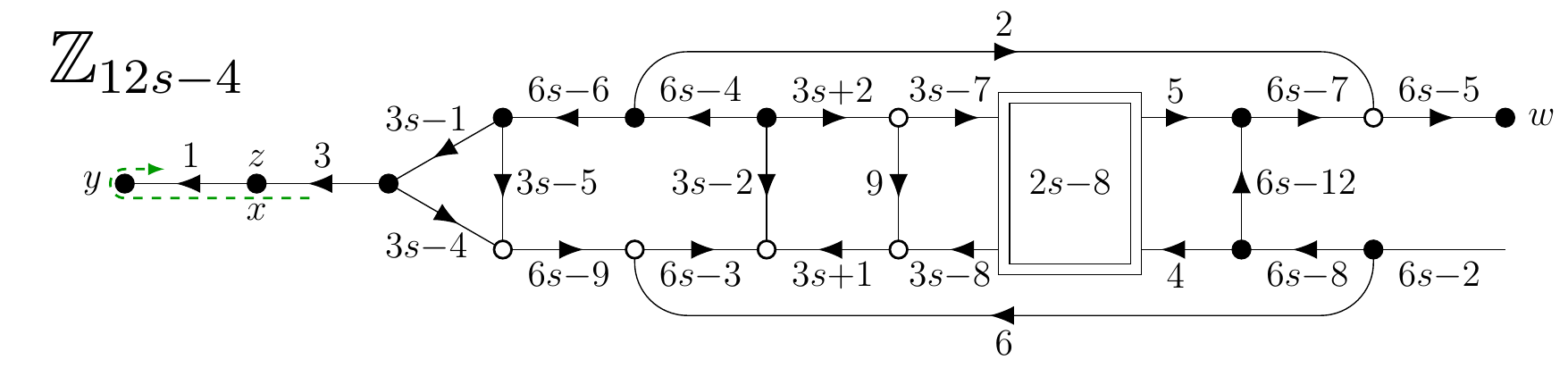}
        \caption{}
        \label{subfig-gen-a}
    \end{subfigure}
    \begin{subfigure}[b]{\textwidth}
        \centering
        \includegraphics[scale=0.85]{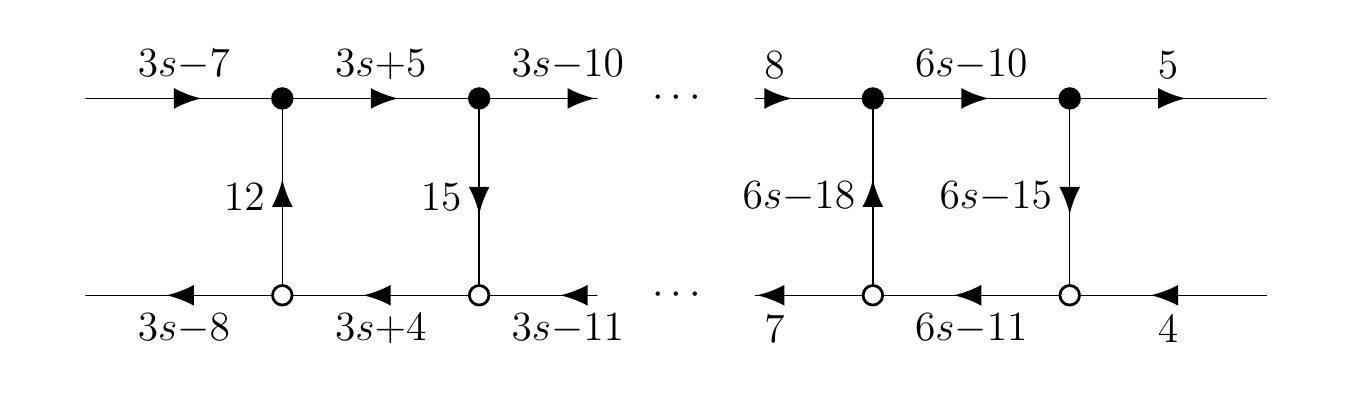}
        \caption{}
        \label{subfig-gen-b}
        \end{subfigure}
    \caption{The fixed part of the current graph (a) contains the salient currents for adding one handle, and the simple ladder (b) inside the box varies in length depending on $s$. The ``rungs'' alternate in direction and form an arithmetic sequence.}
    \label{fig-current-gen}
\end{figure}

For $s \geq 4$, consider the current graph in Figure~\ref{fig-current-gen}, which uses the cyclic group $\Z_{12s-4}$. By examining its log near the vortices and the curved arcs in part (a), we find that the rotation at vertex $0$ is of the form
\begin{equation*}
\setlength{\arraycolsep}{4pt}
\begin{array}{rcccccccccccccccccccc}
0. & x & 1 & y & 12s{-}5 & z & \dots & 6s{+}2 & 2 & 6s{+}3 & \dots & 6s{-}2 & 6 & 6s{-}3 & \dots & 6s{-}5 & w & \dots 
\end{array}
\tag{\textasteriskcentered}
\end{equation*}

\begin{figure}[!ht]
    \centering
    \includegraphics[scale=0.9]{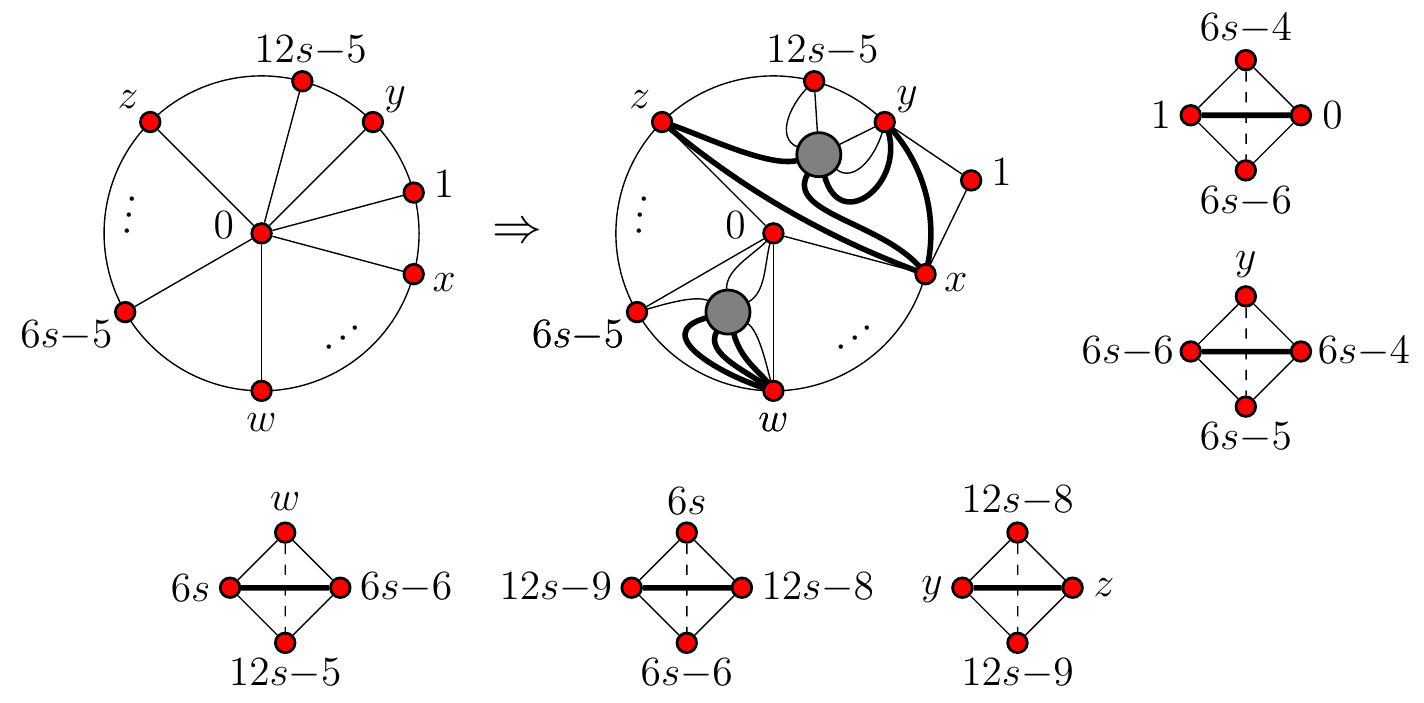}
    \caption{The embedding near vertex 0 is augmented using a handle, which is represented by excising two disks and identifying their boundaries. For the edge flips, the dashed edges are replaced with the thick solid edges. } 
    \label{fig-handle}
\end{figure}

Figure~\ref{fig-handle} illustrates how this partial information allows us to add the missing edges using edge flips and a handle. In particular, adding the edge $(x,y)$ at the cost of $(0,1)$ allows us to install a handle that merges three faces containing the four lettered vertices. We then add the missing edges near this handle or via sequences of edge flips.

The same procedure can be applied for $s=3$ using the current graph in Figure~\ref{fig-case3}, where there is a negligible difference in one of the edge flips. In particular, the log near $6$ is reversed from (\textasteriskcentered), so one of the quadrilaterals in Figure~\ref{fig-handle} is mirrored. The remaining cases $s = 1, 2$ are handled by the symmetric embeddings detailed in Appendix~\ref{sec-app}. 
\end{proof}

\begin{remark}
The arcs labeled $9$ and $6s{-}12$ in Figure~\ref{fig-current-gen}(a) extend the arithmetic sequence in Figure~\ref{fig-current-gen}(b), but unfortunately the rotations assigned to their endpoints must differ from the pattern in Figure~\ref{fig-current-gen}(b). Thus, generalizing this family of current graphs to the $s=3$ case is impossible---the corresponding graph actually does not have any one-face embedding. 
\end{remark}

A naive calculation on the number of nonisomorphic embeddings already sharpens the exponential bound in Korzhik~\cite{Korzhik-Index4} by approximately a factor of 2 in the base. It would be desirable to tighten this bound by developing techniques for distinguishing nearly-regular triangulations. 

\begin{corollary}
The number of nonisomorphic triangulations of $K_{12s}$ is at least $4^{s-o(s)}$, where $o(s)$ denotes a positive function which grows sublinearly in $s$.
\end{corollary}
\begin{proof}
A consequence of Lemma 2 of Korzhik and Voss~\cite{KorzhikVoss} shows that given a one-face embedding of a graph with an $m$-rung ladder, there are at least $2^m$ distinct one-face embeddings, all satisfying property (\textasteriskcentered) (where the subsequences may appear in a different order), that arise from changing the rotations at the vertices in the ladder. Any isomorphism between two triangulations of $K_n$ is uniquely determined by the image of a single edge, and whether the mapping is orientation-reversing. Thus, in a collection of distinct triangulations of $K_{12s}$, any given isomorphism type can only appear at most $2(12s)(12s-1)$ times, so there are at least
$$\frac{2^{2s-8}}{2(12s)(12s-1)} = 4^{s-o(s)}$$
nonisomorphic triangulations.
\end{proof}

\bibliographystyle{alpha}
\bibliography{../biblio}

\appendix

\section{Triangulations of $K_{12}$ and $K_{24}$ using cyclic groups}\label{sec-app}

For completeness, we list the embedding of $K_{12}$ described in Pengelley and Jungerman~\cite{Pengelley-Index4} and follow their approach to give an explicit triangulation of $K_{24}$ using an abelian group (see Figure~\ref{fig-k24}). The latter has not previously appeared in the literature. Both are index 4 solutions using the cyclic groups $\Z_{12}$ and $\Z_{24}$, respectively: to generate the rotation at vertex $k$, take the log $[k \bmod{4}]$ and add $k$ to each element. For more information on index 4 current graphs and the additional difficulties they incur, see Pengelley and Jungerman~\cite{Pengelley-Index4} or Korzhik~\cite{Korzhik-Index4}. 

\begin{figure}[!ht]
    \centering
    \includegraphics[width=\textwidth]{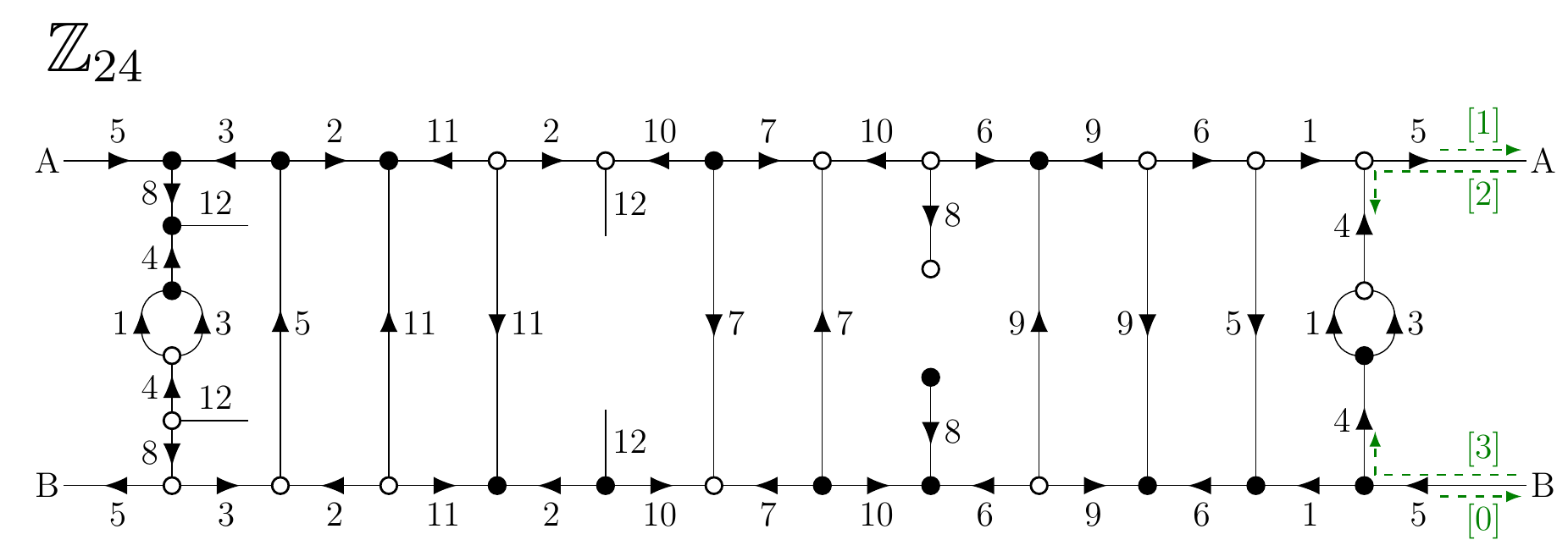}
    \caption{An index 4 current graph for $K_{24}$, produced from the ``two-table'' method of Pengelley and Jungerman~\cite{Pengelley-Index4}. The ends labeled ``A'' and ``B'' are identified.} 
    \label{fig-k24}
\end{figure}

\begin{table}[ht!]
\centering
{
$\begin{array}{rrrrrrrrrrrrrr}
\lbrack0\rbrack. & 11 & 8 & 9 & 1 & 4 & 3 & 6 & 2 & 7 & 5 & 10 \\
\lbrack1\rbrack. & 1 & 4 & 3 & 11 & 8 & 9 & 6 & 10 & 5 & 7 & 2 \\
\lbrack2\rbrack. & 11 & 1 & 2 & 7 & 5 & 10 & 4 & 8 & 6 & 9 & 3 \\
\lbrack3\rbrack. & 1 & 11 & 10 & 5 & 7 & 2 & 8 & 4 & 6 & 3 & 9
\end{array}$
}
\caption{The logs of Pengelley and Jungerman~\cite{Pengelley-Index4} for embedding $K_{12}$.}
\label{tab-k12}
\end{table}

\begin{table}[ht!]
\centering
{
$\arraycolsep=3.6pt\begin{array}{rrrrrrrrrrrrrrrrrrrrrrrrrrrrrrrrrrrrrrrrr}
\lbrack0\rbrack. & 19 & 16 & 4 & 1 & 21 & 20 & 12 & 8 & 3 & 5 & 2 & 13 & 11 & 22 & 10 & 17 & 7 & 14 & 6 & 15 & 9 & 18 & 23 \\
\lbrack1\rbrack. & 5 & 8 & 20 & 23 & 3 & 4 & 12 & 16 & 21 & 19 & 22 & 11 & 13 & 2 & 14 & 7 & 17 & 10 & 18 & 9 & 15 & 6 & 1 \\
\lbrack2\rbrack. & 19 & 20 & 21 & 1 & 4 & 23 & 5 & 6 & 15 & 9 & 18 & 8 & 16 & 10 & 17 & 7 & 14 & 12 & 2 & 13 & 11 & 22 & 3 \\
\lbrack3\rbrack. & 5 & 4 & 3 & 23 & 20 & 1 & 19 & 18 & 9 & 15 & 6 & 16 & 8 & 14 & 7 & 17 & 10 & 12 & 22 & 11 & 13 & 2 & 21
\end{array}$
}
\caption{Logs for generating a triangulation of $K_{24}$ using the group $\Z_{24}$.}
\label{tab-k24}
\end{table}

\end{document}